\documentclass{birkjour}

\usepackage{amssymb,bbm,mathrsfs}

 \newtheorem{thm}{Theorem}[section]
 \newtheorem{cor}[thm]{Corollary}
 \newtheorem{lem}[thm]{Lemma}
 \newtheorem{prop}[thm]{Proposition}
 \theoremstyle{definition}
 \newtheorem{defn}[thm]{Definition}
 \theoremstyle{remark}
 \newtheorem{rema}[thm]{Remark}
 \theoremstyle{rule}
 \newtheorem{rul}[thm]{Rule}
 
 \theoremstyle{example}
 \newtheorem{exmp}[thm]{Example}
 \numberwithin{equation}{section}

\newcommand\blfootnote[1]{%
  \begingroup
  \renewcommand\thefootnote{}\footnote{#1}%
  \addtocounter{footnote}{-1}%
  \endgroup
}

\begin{document}

\title[]
{\centering{An Algebraic Approach to \\ Degenerate Bernoulli Numbers}}

\author[]{\centering{N. Uday Kiran and Sampath Lonka}}

\address{Department of Mathematics and Computer Science\\ Sri Sathya Sai Institute of Higher Learning\\
INDIA}
\email{nudaykiran@sssihl.edu.in\\
\hspace*{0.9cm} sampathlonka@sssihl.edu.in}
\subjclass{Primary 11B68 Secondary 11B83}
\keywords{ Degenerate Bernoulli Number, Sum of Product formula, Galois Fields, Palindrome Polynomials of a Darga, Circulant Matrices}

\blfootnote{Dedicated to Bhagawan Sri Sathya Sai Baba} 

\begin{abstract}
In this work we study the properties of a new algebraic variant of the degenerate Bernoulli polynomial $\tilde{\beta}_{k}(m,x)$ and study the corresponding degenerate Bernoulli number $\tilde{\beta}_{k}(m,1)=m^{k}\beta_{k}(1/m)$, where $\beta_{k}(\lambda), \lambda\neq 0$ is the standard degenerate Bernoulli number. Our approach relies on a new algebraic framework for generating functions and the action of a symbolic evaluation function on powers of polynomials. We show that $\tilde{\beta}_{k}(m,x)$ displays surprising links with other mathematical objects (such as Circulant matrices and Galois fields) and enjoys many interesting algebraic, symmetric and dynamical properties that could be deployed to perform efficient computations.  
\end{abstract}

%Todolist
%1) p=7 m=4,8,16 (making the precise the periodicity)
%2) umbral calculus
%3) quantifying \eta 
%4) explain $O_{1\times (m-1)}$
%%% ----------------------------------------------------------------------
\maketitle
%%% ----------------------------------------------------------------------
%\tableofcontents
\section{Introduction}

Degenerate Bernoulli numbers $\beta_{k}(\lambda)$, where $\lambda\neq 0$, were proposed by Carlitz \cite{Carlitz}, and are of contemporary interest (see \cite{kim,howard}). He defined them by means of the generating function 
\begin{equation}\label{carlitz}
\frac{t}{(1+\lambda t)^{1/\lambda}-1}=\sum_{k=0}^{\infty}\beta_{k}(\lambda)\frac{t^{k}}{k!}.
\end{equation}
These numbers are related to the Bernoulli numbers by $\displaystyle\lim_{\lambda \rightarrow 0}\beta_{k}(\lambda)=B_{k}$.

In this paper, our interest is to study the properties of $\tilde{\beta}_{k}(m)=m^{k}\beta_{k}(1/m)$ for $m$ a positive integer. As observed in \cite{uk}, by setting $x=(1+\lambda t)$ and $\lambda=1/m$ in (\ref{carlitz}) we have the simpler generating function 
$$
m\frac{1-x}{1-x^{m}}=\frac{m}{\Psi_{m}(x)}=\sum_{k=0}^{\infty}\frac{(-1)^{k}}{k!}\tilde{\beta}_{k}(m)(1-x)^{k},
$$
where the polynomial $\Psi_{m}(x)=\displaystyle\sum_{j=0}^{m-1}x^{j}$. This generating function can be studied through a special $m-2$ degree polynomial which emerges as a remainder of powers of a polynomial
\begin{equation}\label{fkm}
f^{(m)}_{k}(x)=\frac{(-1)^{k}}{k!}\tilde{\beta}_{k}(m,x) = \left(\frac{1}{m}x\Psi_{m}'(x)\right)^{k} \textnormal{ rem }\Psi_{m}(x),
\end{equation}
where rem stands for the polynomial remainder operator. It should be noted that the $k^{th}$ polynomial is obtained by taking remainder of the $k^{th}$ power of a polynomial through another polynomial. Thus, these polynomials are obtained through polynomial manipulation and are different from other the degenerate Bernoulli polynomials defined in the literature (see \cite{Carlitz,kim,young,howard} and the references therein).

Upon substitution of $x=1$ in (\ref{fkm}) we have 
$$
f^{(m)}_{k}(1)=\frac{(-1)^{k}}{k!}\tilde{\beta}_{k}(m,1)=\frac{(-1)^{k}}{k!}\tilde{\beta}_{k}(m).
$$
Gessel \cite{Gessel} showed that the number $\tilde{\beta}_{k}(m)$ can be expressed through a trigonometric sum and obtained an explicit formula.  

\begin{thm}[Gessel \cite{Gessel}]\label{fjm_polynomial}
\begin{equation}\label{fjm}
\tilde{\beta}_{k}(m) = \left\{\begin{array}{@{}l@{\thinspace}l}
        \displaystyle  \frac{1-m}{2} & \textnormal{   if   } k=1 \\
     \displaystyle  k\sum^{k}_{j=2}(-1)^{k-j}\frac{B_{j}}{j}\textnormal{stirl}(k-1,j-1)(m^{j}-1)  & \textnormal{   if   } k\geq 2. \\
     \end{array}\right.
\end{equation}
\end{thm}

In \cite{uk}, the first author showed that $\tilde{\beta}_{k}(m)$ is a polynomial in $m$ only through algebraic considerations. Further, it was shown in \cite{uk}, that there is a sum of product based formula for the polynomial part of the Sylvester Denumerants (aka Restricted Partitions) given by 
\begin{eqnarray}
\label{W1}W_{1}(t;\textbf{A}) &=& \frac{1}{a_{1}\cdots a_{r}}\sum^{r-1}_{j=0}\sum_{j_{1}+\cdots+j_{r}=j}\begin{pmatrix}t+r-j-1 \\ t\end{pmatrix}f_{j_{1}}^{(a_{1})}(1)\cdots f_{j_{r}}^{(a_{r})}(1)\\
\nonumber &=& \frac{1}{a_{1}\cdots a_{r}}\sum^{r-1}_{j=0}\sum_{j_{1}+\cdots+j_{r}=j}(-1)^{j}\begin{pmatrix}t+r-j-1 \\ t\end{pmatrix}\frac{\tilde{\beta}_{j_{1}}(a_{1})\cdots \tilde{\beta}_{j_{r}}(a_{r})}{j_{1}!\cdots j_{r}!},
\end{eqnarray}
where $\textbf{A}=(a_{1},\cdots,a_{r})$ is a tuple of positive integers. Theses numbers have also been found useful in identities of falling factorials (for instance, see \cite{young,howard,zhang} and the references therein).

These above observations motivate us to study the properties of the degenerate Bernoulli numbers with a view to efficient computation of  $W_{1}(t;\textbf{A})$. In this direction, we study the polynomial (\ref{fkm}) from algebraic, symmetric and dynamical viewpoints. Our contributions in this paper are the following results on $\tilde{\beta}_{k}(m,x)/k!$:
\begin{enumerate}
    \item Periodicity in a modulo $p$ (a prime) and $(m,p)=1$. As a consequence of this property and the reciprocity law we show that several $\tilde{\beta}_{k}(m)$ vanish mod $p$.
    \item For $m$ odd number, $\tilde{\beta}_{md+2}(m,x)$ is a palindrome of darga $(m-2)$ if $d$ is even and anti-palindrome of darga $(m-2)$ if $d$ is odd. Using this fact, we can provide an algebraic proof for the well know fact $\tilde{\beta}_{md+2}(m)=0$ for $d$ odd given in \cite{young}. 
    \item A sum of products formula holds for relatively coprime $m_{1}$ and $m_{2}$: 
    $$
    \sum_{j_{1}+j_{2}=k}\frac{\tilde{\beta}_{j_{1}}(m_{1})}{j_{1}!}\frac{\tilde{\beta}_{j_{2}}(m_{2})}{j_{2}!}=\sum^{m_{1}+m_{2}}_{j=0}a_{j}\frac{\tilde{\beta}_{j}(m_{1})}{j!}+b_{j}\frac{\tilde{\beta}_{j}(m_{2})}{j!}.
    $$
    For the case $m_{1}=m_{2}$ the sum of products result is given in \cite{zhang}.
    \item The set of all polynomials $\tilde{\beta}_{k}(m)$ for $k\in \mathbb{Z}$ form an infinite abelian group under multiplication mod $\Psi_{m}(x)$. In $\mathbb{F}_{p}$, this group becomes finite and is isomorphic to a subgroup of $\mathbb{F}_{p^{m-1}}^{*}$. Moreover, under certain mild conditions, the subgroup is equidistributed in $\mathbb{F}_{p^{m-1}}^{*}$.
    \item The matrix representation of the $m-2$ degree polynomials $\tilde{\beta}_{k}(m,x)/k!$ for varying $k$ are related to powers of a certain Circulant matrix \cite{davis}, which in turn has a special feature of being diagonalizable by DFT matrix. We call these matrices the Degenerate Bernoulli Matrices.
\end{enumerate}

Our approach through polynomials can be considered as an alternative approach to the matrix approach to studying sequences satisfying linear recurrence relations. This is presented in Section \ref{sec_matrix}.

\section{An Algebraic Framework for Sequences Generated by Rational Polynomials}\label{sec_framework}

Given two non-zero polynomials $p(x),q(x)\in \mathbb{Q}[x]$ we say $p(x)/q(x)$ is a rational polynomial. Further, if $\textnormal{deg}(p)<\textnormal{deg}(q)$ then we say that it is a proper rational polynomial. In order to provide an algebraic framework for sequences generated by a rational polynomial we define an operator from commutative algebra.

\begin{defn}\label{eval}
Given non-constant polynomials $r(x), s(x), a(x)\in \mathbb{Q}[x]$ and $s(x),a(x)$ are relatively prime, we define the evaluation of the fraction $\frac{r(x)}{s(x)}$ with respect to $a(x)$ 
$$
\textnormal{eval}\left(\frac{r(x)}{s(x)}; a(x)\right)
= \alpha(x)r(x) \textnormal{ rem }a(x),
$$ 
where $\alpha(x) s(x)\equiv 1 \textnormal{ mod }a(x)$ and $\textnormal{rem}$  is the polynomial remainder operator.
\end{defn}

In simple terms, the eval of $r(x)/s(x)$ takes the remainder of $r(x)$ multiplied by the inverse of $s(x)$ modulo $a(x)$. The eval operator is based on the principles of substitution and localization. The following properties can be easily proved.
\begin{lem} 
Given $a(x), r_{i}(x), s_{i}(x)\in \mathbb{Q}[x]$ and $\textnormal{gcd}(s_{i}(x),a(x))=1$ for $i=0,1$. The following are some properties of $\textnormal{eval}$ function:

\begin{enumerate}
\item $\textnormal{eval}\left(r_{0};a(x)\right)=r_{0}(x)\textnormal{ rem }a(x).$
\item $\textnormal{eval}\left(\alpha_{0} r_{0}(x)+\alpha_{1} r_{1}(x); a(x)\right)=\alpha_{0}\textnormal{eval}\left( r_{0}(x);a(x)\right)+\alpha_{1}\textnormal{eval}\left(r_{1}(x);a(x)\right)$.
\item $\textnormal{eval}\left(\frac{r_{0}(x)r_{1}(x)}{s_{0}(x)s_{1}(x)};a(x)\right)=\textnormal{eval}\left(\frac{r_{0}(x)}{s_{0}(x)};a(x)\right)\textnormal{eval}\left(\frac{r_{1}(x)}{s_{1}(x)}; a(x)\right).$
\item $\textnormal{eval}\left(\frac{r_{0}(x) - p_{0}(x)a(x)}{s_{0}(x)- q_{0}(x) a(x)};a(x)\right)=\textnormal{eval}\left(\frac{r_{0}(x)}{s_{0}(x)};a(x)\right)$ \textnormal{for some} $p_{0}(x),q_{0}(x)\in \mathbb{Q}[x].$
\end{enumerate}
\end{lem}

For a rigorous treatment of this operator and its properties we refer to \cite{uk}. The operator eval was shown to be an effective tool to handle partial fractions. The following result was indeed proved. 

\begin{thm}[Extended Cover-Up Method, \cite{uk}]\label{cover_up}
Let $p_{1}(x),\cdots, p_{n}(x)\in \mathbb{Q}[x]$ be pairwise relatively prime polynomials and let $f(x)\in \mathbb{Q}[x]$ satisfying $deg(f)< deg(p_{1}\cdots p_{k})$. Then, we have the following identity 
$$
f(x)\prod_{j=1}^{k}\frac{1}{p_{j}(x)}=\sum^{k}_{j=1}
\frac{f(x)\textnormal{eval}\left(\displaystyle\prod_{\substack{i=1\\ i\neq j}}^{n}\frac{1}{p_{i}(x)};p_{j}(x)\right)}{p_{j}(x)}.$$
\end{thm}
As an immediate corollary we prove the following useful partial fraction decomposition result.

\begin{lem}\label{k_1_case_lem}
Given $p(x),q(x)\in \mathbb{Q}[x]$ with $\textnormal{deg}(p)<\textnormal{deg}(q)$ and $\textnormal{gcd}(p(x),q(x))=1$. Suppose $p(x)$ and $q(x)$ are non-vanishing at $x=a$ we have the partial fraction identity
\begin{equation}\label{k_1_case}
    \frac{p(x)}{(x-a)q(x)}=\frac{p(a)}{q(a)(x-a)}+\frac{\textnormal{eval}\left(p(x)/(x-a);q(x)\right)}{q(x)},
\end{equation}
where 
\begin{equation}\label{k_1_case_eval}
\textnormal{eval}\left(\frac{p(x)}{x-a};q(x)\right)=\frac{q(a)p(x)-p(a)q(x)}{q(a)(x-a)}.
\end{equation}
\end{lem}
\begin{proof}
By the Extended Cover-Up method we have 
$$
\frac{p(x)}{(x-a)q(x)}=\frac{\textnormal{eval}\left(p(x)/q(x);x-a\right)}{(x-a)}+\frac{\textnormal{eval}\left(p(x)/(x-a);q(x)\right)}{q(x)}.
$$

Since $q(a)\neq 0$, we have $q(x)$ and $(x-a)$ to be coprimes. By Bezout's identity, there exist $\alpha(x),\beta(x)$ such that $\alpha(x)q(x)+\beta(x)(x-a)=1$. Without loss of generality we may assume that $\alpha(x)$ is a remainder when divided by $(x-a)$, i.e., $\alpha(x)$ is a constant. Otherwise, we can write $\alpha(x)=\tilde{\alpha}(x)(x-a)+r(x)$ where $\textnormal{deg}(r)=0$ and express the above equation as $r(x)q(x)+(\beta(x)+\tilde{\alpha}(x))(x-a)=1$. 

Substituting $x=a$ we obtain $\alpha(x)=1/q(a)$. Further, by the definition of eval and the Remainder Theorem we have $\textnormal{eval}(p(x)/q(x);(x-a))=p(a)/q(a)$. Upon rearranging terms in (\ref{k_1_case}) we obtain (\ref{k_1_case_eval}).
\end{proof}

The above lemma is a general partial fraction formula that can be inductively extended to higher powers case.  

\begin{lem}\label{lem_pf}
Given two polynomials $p(x),q(x)$ with $\textnormal{deg}(p)<\textnormal{deg}(q)$ and both $p(a)$ and $q(a)$ are non-zero. For $k\geq 1$ the following partial fraction holds
\begin{equation}
\frac{p(x)}{(x-a)^{k}q(x)} =\frac{1}{q(a)}\sum_{j=0}^{k-1}\frac{g_{j}(a)}{(x-a)^{k-j}}+\frac{g_{k}(x)}{q(x)},
\end{equation}
where $g_{k}(x)=\textnormal{eval}(p(x)/(x-a)^{k};q(x))$ and $g_{0}(x)=p(a)$. 
\end{lem}
\begin{proof}
We prove the result by induction on $k$. The case $k=1$ is given in Lemma \ref{k_1_case_lem}. Suppose the result holds for $k=n$, i.e.  
$$
\frac{p(x)}{(x-a)^{n}q(x)} =\frac{1}{q(a)}\sum_{j=0}^{n-1}\frac{g_{j}(a)}{(x-a)^{n-j}}+\frac{g_{k}(x)}{q(x)}.
$$
Dividing both sides of the above equation by $(x-a)$ we have 
$$
\frac{p(x)}{(x-a)^{n+1}q(x)} =\frac{1}{q(a)}\sum_{j=0}^{n-1}\frac{g_{j}(a)}{(x-a)^{n-j+1}}+\frac{g_{n}(x)}{(x-a)q(x)}.
$$
Applying Lemma \ref{k_1_case_lem} on the last term we have 
$$
\frac{g_{n}(x)}{(x-a)q(x)}=\frac{g_{n}(a)}{q(a)(x-a)}+\frac{\textnormal{eval}(g_{n}(x)/(x-a);q(x))}{q(x)}.
$$
It is easy to see that $g_{n+1}(x)=\textnormal{eval}(g_{n}(x)/(x-a);q(x))$.

\end{proof}

Similar to the partial fractions formula we can also prove a result on the long division.   
\begin{lem}\label{lem_division}
Given two polynomials $p(x),q(x)$ with $\textnormal{deg}(p)<\textnormal{deg}(q)$ and both $p(a)$ and $q(a)$ are non-zero. For $k\geq 1$ the following repeated division result holds
\begin{equation}
\frac{p(x)(x-a)^{k}}{q(x)} =\frac{-1}{q(a)}\sum_{j=1}^{k}g_{-j}(a)(x-a)^{k-j}+\frac{g_{-k}(x)}{q(x)},
\end{equation}
where $g_{k}(x)=\textnormal{eval}(p(x)/(x-a)^{k};q(x))$. 
\end{lem}

\begin{proof}
Let $p(x)(x-a)=s(x)q(x)+r(x)$ for $\textnormal{deg}(r)<\textnormal{deg}(q)$. Then, by $\textnormal{deg}(p)<\textnormal{deg}(q)$ we can easily deduce that $\textnormal{deg}(s)=0$, i.e., $s(x)=c$ a constant. By the definition of eval operator we have $r(x)=\textnormal{eval}(p(x)(x-a);q(x))=g_{-1}(x)$. Thus, upon substitution by $x=a$ and $r(x)=g_{-1}(x)$ into the above equation we have
$$
0=cq(a)+g_{-1}(a),\ \textnormal{ implies } c=-\frac{g_{-1}(a)}{q(a)}.
$$
Inducting on $k$ one can easily prove the result. 
\end{proof}

Equipped with the above lemmas we now construct an algebraic framework for sequences generated by rational polynomials. A generating function $F(x-a)$ about the point $x-a=0$ is the formal power series 
$$
F(x-a)=\sum^{\infty}_{k=0}b_{k}(x-a)^{k}.
$$
whose coefficients form the sequence $(b_{k})$.

If the sequence $(b_{k})$ satisfies a finite linear recurrence relation then one can show that $F(x-a)$ can be expressed as a rational polynomial $p(x)/q(x)$ where $p(x),q(x)$ and $q(a)\neq 0$ \cite{beck1}. In this work, we only consider sequences generated by proper rational polynomials. 

A formal power series expansion of $F(x)$ can also be obtained about the point $x-a=\infty$ which yields the sequence $(b_{-k})$ with negative indices. In order to obtain the coefficients with negative index 
$$
F^{o}(x-a)=\sum_{k=1}^{\infty} b_{-k}(x-a)^{k}.
$$
The coefficients of the negative index is given by \cite{beck1}
$$
F^{o}(x-a)=-F\left(\frac{1}{x-a}\right).
$$
Based on the above lemmas we now define the \textit{generating polynomial of a sequence} given by the generating $p(x)/q(x)$ with a formal power series centered about $x=a$ by 
$$
g_{k}(x)=\textnormal{eval}\left(p(x)(x-a)^{-k};q(x)\right)
$$
for $k\in \mathbb{Z}$ and $b_{k}=g(a)/q(a)$. 

We can prove the following result immediately from Lemma \ref{lem_pf} and Lemma \ref{lem_division}.

\begin{thm}[Generating Polynomials]
The generating function $p(x)/q(x)$ (with $q(a)\neq 0$) for the sequence $(b_{k})$ centered about $x=a$, i.e.,  
$$
\frac{p(x)}{q(x)}=\sum^{\infty}_{k=0}b_{k}(x-a)^{k}
$$
has the generating polynomial 
$$
g_{k}(x)=\textnormal{eval}\left((x-a)^{-k}p(x);q(x)\right)
$$
and $b_{k}=g_{k}(a)/q(a)$. Furthermore, the negative indexed coefficient associated with the recurrence relation on the coefficients $(b_{k})$ can also be obtained using $g_{k}(x)$ by considering negative values of $k$. 
\end{thm}

The relation between the positive and negative index polynomials can be stated succinctly for a special case.  

\begin{cor}[Reciprocity Law]
If $p(x)=1$ then the generating polynomial $g_{k}(x)$ and $g_{-k}(x)$ are reciprocal of each other mod $q(x)$, i.e.,
$$
g_{k}(x)g_{-k}(x)=1\mod\ q(x).
$$ Therefore, the set of generating polynomials \begin{equation}
G=\left\{\textnormal{eval}\left((x-a)^{-k};q(x)\right):k\in \mathbb{Z}\right\}
\end{equation}
forms an infinite abelian group under multiplication modulo $q(x)$.
\end{cor}

The generating polynomials also satisfy a recurrence relation mod $q(x)$. Indeed, suppose $q(x)=\sum_{j=0}^{s}\alpha_{j}(x-a)^{j}$. Consider the expression mod $q(x)$
$$
\sum^{s}_{j=0}\alpha_{j}g_{k+s-j}(x)= \sum_{j=0}^{s}\alpha_{j}\textnormal{eval}\left((x-a)^{k+s-j};q(x)\right).
$$
By the linearity of the eval operator we have 
\begin{eqnarray}
\nonumber \sum^{s}_{j=0}\alpha_{j}g_{k+s-j}(x)&=&\textnormal{eval}\left(\sum_{j=0}^{s}\alpha_{j}(x-a)^{k+s-j};q(x)\right)\\
\nonumber &=& \textnormal{eval}\left(q(x)(x-a)^{k};q(x)\right)\\
\nonumber &=& \beta(x)q(x) \textnormal{ rem }q(x) =0,
\end{eqnarray}
where $\beta(x)(x-a)^{k}=1\mod q(x)$ if $k$ is negative; otherwise, $\beta(x)=(x-a)^{k}$.

Before we proceed to the degenerate Bernoulli polynomials we demonstrate through an example the generality of our algebraic framework. We consider the Fibonacci sequence. 

\begin{exmp}[Fibonacci Sequence]
It is well known that the Fibonacci sequence is given by the generating function 
$$
\frac{x}{1-x-x^{2}}=\sum^{\infty}_{k=0}F_{k}x^{k},
$$
where $p(x)=x, q(x)=1-x-x^2$, and $a=0$.

By the above lemmas we have the Fibonacci polynomial defined as 
$$
F_{k}(x)=\textnormal{eval}\left(x^{-k+1};1-x-x^{2}\right).
$$
Since $(1+x)x = 1 \mod (1-x-x^2)$ we can easily see that 
$$
\textnormal{eval}(x^{-1};1-x-x^{2}) = 1+x.
$$
Thus by the property of eval
$$
\textnormal{eval}(x^{-k};1-x-x^{2}) = (1+x)^k \textnormal{ rem } (1-x-x^2).
$$

Now, by a direct computation one can easily check that $F_{0}(x)=x, F_{1}(x)=1, F_{2}(x)=1+x, F_{3}(x)=2+x$ which when evaluated at $x=0$ we obtain the required sequence. Moreover, for negative indices we have $F_{-1}(x)=1, F_{-2}(x)=x,F_{-3}(x)=1-x$ which again when evaluated at $x=0$ give us the Fibonacci numbers with negative indices. 

Furthermore,
$$
\textnormal{eval}(x^{-k};1-x-x^{2}) = (1+x)^{k} \textnormal{ rem } (1-x-x^{2}).
$$
Therefore, we can write a recursive relation 
$$
F_{n+1}(x) = (x+1)F_{n}(x)\textnormal{ rem } (1-x-x^2),
$$
and $F_{0}(x)=x$. If we set $F_{n}(x)=a_{0}+a_{1}x$ then we can compute $F_{n+1}(x)=(a_{0}+a_{1})+a_{0}x$. A matrix realization of this transformation is given by 
$$
\begin{pmatrix}
1 & 1 \\ 1 & 0
\end{pmatrix}\begin{pmatrix}
a_{0} \\ a_{1}
\end{pmatrix}=\begin{pmatrix}
a_{0}+a_{1} \\ a_{0}
\end{pmatrix},
$$
which, although derived by a different approach, is the well known Fibonacci Q-matrix \cite{fib}.
\end{exmp}

\section{The Polynomial $\Psi_{m}(x)$}

For $m$ a positive integer. We denote the polynomial
$$
\Psi_{m}(x)=\frac{1-x^{m}}{1-x}.
$$
The eval operator behaves in a very simpler manner with respect to $\Psi_{m}(x)$. 

\begin{rul}[Substitution Rule]\label{subs_rule_1}
\begin{equation}\label{rem_psi}
\textnormal{eval}(x^{j};\Psi_{m}(x))=x^{j} \textnormal{ rem } \Psi_{m}(x)= \left\{\begin{array}{@{}l@{\thinspace}l}
       x^{j\%m} & \textnormal{   if   } j\%m \neq m-1 \\
       -\sum^{m-2}_{i=0}x^{i}  & \textnormal{   if   } j\%m=m-1. \\
     \end{array}\right.
\end{equation}
where $\%$ operator is the remainder in the ring of integers. 
\end{rul}
\begin{proof}[Proof.]
Suppose $j=md+r$, for $0\leq r<m$. Then, the result follows from 
$
x^{md+r}=((x^{m}-1)+1)^{d}x^{r}=q(x)(x^{m}-1)+x^{r}=q(x)(x-1)\Psi_{m}(x)+x^{r}.
$
\end{proof}

In \cite{uk}, the first author showed that $\Psi_{m}(x)$ is naturally related to the degenerate Bernoulli numbers. Indeed, the degenerate Bernoulli numbers $\beta_{k}(\lambda)$, where $\lambda\neq 0$, is given by the generating function 
\begin{equation}\label{carlitz1}
\frac{t}{(1+\lambda t)^{1/\lambda}-1}=\sum_{k=0}^{\infty}\beta_{k}(\lambda)\frac{t^{k}}{k!}.
\end{equation}
By setting $x=(1+\lambda t)$ and $\lambda=1/m$ in (\ref{carlitz1}) we obtain the generating function $$
m\frac{1-x}{1-x^{m}}=\frac{m}{\Psi_{m}(x)}=\sum_{k=0}^{\infty}(-1)^{k}\frac{\tilde{\beta}_{k}(m)}{k!}(1-x)^{k}.
$$
Also, note that $B_{k}=\beta_{k}(1/m)$ as $m\rightarrow \infty$.

By the framework developed in Section \ref{sec_framework} for $p(x)/q(x)=1/\Psi_{m}(x)$ and obtaining power series about $x=1$ we have 
$$
f_{k}^{(m)}(x)=\textnormal{eval}\left((1-x)^{-k};\Psi_{m}(x)\right),
$$
and $f^{(m)}_{k}(1)=(-1)^{k}\tilde{\beta}_{k}(m)/k!$. Furthermore, we can also define the above polynomial for negative $k$. So, we have the following result. 

\begin{prop}
$$
f_{-k}^{(m)}(1)=0 \textnormal{ for } 1\leq k < (m-1).
$$
\end{prop}
\begin{proof}
For $k$ with $1\leq k<(m-1)$ we have 
$$
f_{-k}^{(m)}(x)=\textnormal{eval}((1-x)^{k};\Psi_{m}(x))= (1-x)^{k}.
$$
The above equations holds because the degree of $(1-x)^{k}$ is less than that of $\Psi_{m}(x)$. Substituting $x=1$ in the above equation we have the result. 
\end{proof}

\begin{rema}
Note that we can also define $\tilde{\beta}_{k}(-m)$ for positive integer $m$. But the analysis is similar to positive case. Indeed, consider the generating function 
$$
\frac{t}{(1-\frac{t}{m})^{-m}-1}.
$$
Let $x=1-t/m$. Then, $t=m(1-x)$. Upon substitution we have 
\begin{eqnarray}
\nonumber \frac{m(1-x)}{x^{-m}-1}&=&\frac{m(1-x)x^{m}}{1-x^{m}}=\frac{m}{\Psi_{m}(x)}-m(1-x)\\
\nonumber &=&\sum^{\infty}_{k=0}(-1)^{k}\tilde{\beta}_{k}(m)\frac{(1-x)^{k}}{k!}-m(1-x).
\end{eqnarray}
Therefore we have  
$$
\tilde{\beta}_{k}(-m)=\begin{cases} (-1)^{k}\tilde{\beta}_{k}(m), &\text{if }m\neq 1\\ -\tilde{\beta}_{1}(m)-m&\text{if }m=1.\end{cases}
$$
\end{rema}

We now derive the explicit formula (\ref{fkm}).

\begin{lem}
$$
f_{k}^{(m)}(x)=\textnormal{eval}\left(\frac{1}{(x-1)^{k}};\Psi_{m}(x)\right)= \left(\frac{1}{m}x\Psi'_{m}(x)\right)^{k} \textnormal{ rem }\Psi_{m}(x).
$$
\end{lem}
\begin{proof}
Differentiating and multiplying by $x$ both sides of the equation
$$
(x-1)\Psi_{m}(x)=x^{m}-1
$$
we get 
$$
x\Psi'_{m}(x)(x-1)+(x)\Psi_{m}(x)=mx^{m}.
$$
Thus by taking mod $\Psi_{m}(x)$ we have 
$$
x\Psi'_{m}(x)(x-1) = m \mod \Psi_{m}(x).
$$
Raising both sides of the above equation by $k$ we have the result. 
\end{proof}

By Lemma \ref{k_1_case_lem} we can also write 
$$
f_{k}^{(m)}(x)=\left(\frac{\Psi_{m}(x)-m}{x-1}\right)^{k} \textnormal{ rem } \Psi_{m}(x).
$$

\section{Periodicity of $f_{k}^{(m)}(x)=(-1)^{k}\tilde{\beta}_{k}(m,x)/k!$}

In this section, we show that the polynomial $f_{k}^{(m)}(x)$ given in (\ref{fkm}) are periodic modulo a prime $p$ that are coprime to $m$, that is taking the coefficients of $f_{k}^{(m)}(x)$ mod $p$ in $\mathbb{F}_{p}$. A direct consequence is that the infinite abelian group formed by the generating polynomials becomes isomorphic to a multiplicative subgroup of $\mathbb{F}_{p^{m-1}}^{*}$ the Galois Field without the zero element.  

Consider $p$ a prime satisfying $(m,p)=1$. We prove that the polynomial $f_{k}^{(m)}(x)$ is periodic in $k$ with periodicity $p^{s}-1$, where $s$ is smallest positive integer such that $p^{s}=1\mod m$. Notice that, although the field we consider is $\mathbb{F}_{p}$, the condition we impose is based on the mod $m$. Results with this type of role reversal of $p$ and $m$ are generally classified as reciprocity laws \cite{reci_law}. 

In the following result we use the well known fact raising a polynomial $g(x)\in \mathbb{F}_{p}[x]$ to power $p$ gives us $g(x)^{p}=g(x^{p}).$ 

\begin{thm}[Periodicity]
$$
f_{p^{s}-1+n}^{(m)}(x)=f_{n}^{(m)}(x) \mod p,
$$
for $n$ a positive integer and $p^{s}=1\mod m$.
\end{thm}
\begin{proof}
Let $g(x)=\frac{1}{m}x\Psi_{m}'(x)$. Then, by (\ref{fkm}) and the substitution rule in Lemma \ref{subs_rule_1} we have 
\begin{eqnarray}
\nonumber f_{p^{s}-1+n}^{(m)}(x)&=&g(x)^{p^{s}-1+n}\textnormal{ rem }\Psi_{m}(x)= \left(g(x^{p^{s}})g(x)^{n-1}\right)\textnormal{ rem }\Psi_{m}(x)\\
\nonumber &=& \left(g(x^{p^{s}\%m})g(x)^{n-1}\right)\textnormal{ rem }\Psi_{m}(x).
\end{eqnarray}
By the hypothesis we have $p^{s}=1\mod m$. Thus $f^{(m)}_{p^s-1+n}(x)=g(x)^{n}\textnormal{ rem }\Psi_{m}(x)$.
\end{proof}
Thus the degenerate Bernoulli polynomials are periodic with a periodicity $p^{s}-1$ for the smallest $s$ satisfying $p^{s}=1\mod m$. In particular, $f^{(m)}_{p^{s}-1}(x)=f^{(m)}_{0}(x)=1$.

As a consequence of this result we observe that the polynomial in $\mathbb{F}_{p}$
$$
f^{(m)}_{k}(x) = \textnormal{eval}\left(\frac{1}{(1-x)^{k}};\Psi_{m}(x)\right)
$$
can be written as 
$$
f^{(m)}_{k}(x) = \textnormal{eval}\left((1-x)^{p^{s}-1-k};\Psi_{m}(x)\right)=(1-x)^{n(p^{s}-1)-k}\textnormal{ rem }\Psi_{m}(x)
$$
for some large $n$, which amounts to only polynomial evaluation. 

\begin{cor}
$\tilde{\beta}_{k}(m)\equiv 0 \mod p$ for $p^{s}-1-k<m-1$.
\end{cor}

It can be easily shown that the set
$$
G=\{f_{k}^{(m)}(x):k\in \mathbb{Z}\}
$$
is an infinite abelian group under multiplication modulo $\Psi_{m}(x)$. With the above periodicity condition one can show that when we reduce the coefficient of the polynomials in $G$ modulo $p$ then we have a finite abelian group with at most $p^{s}-1$ elements where $s$ is the smallest positive integer such that $p^{s}=1\mod m$. 

\begin{cor}
Let 
$
G_{p}=\{f_{p}(x): f(x)\in G\},
$
where $f_{p}(x)$ is the polynomial obtained by taking $\mod p$ of the coefficients of $f(x)$. Then, $G_{p}$ is a subgroup of $\mathbb{F}^{*}_{p^{m-1}}$, where $s$ is the smallest positive integer such that $p^{s}=1\mod m$.
\end{cor}

By \cite{bourgain}, if $p^{m-1}$ is sufficiently large, we can in fact show that $G_{p}$ is equidistributed in $\mathbb{F}_{p^{m-1}}$. The elements of $G_{p}$ span whole of $\mathbb{F}_{p^{m-1}}$ over $\mathbb{F}_{p}$. Indeed, the following elements of $G_{p}$ 
$$
\textnormal{eval}\left((1-x)^{k}; \Psi_{m}(x)\right)=(1-x)^{k}\textnormal{  for  } 0\leq k\leq m-2
$$
form a linearly independent set with $m-1$ elements. Therefore, no proper subfield $S$ of $\mathbb{F}_{p^{m-1}}$ contains $G_{p}$. Hence, one can find a $\eta\in(0,1]$ such that 
$$
|S\cap G_{p}|\leq |G_{p}|^{1-\eta},
$$
where $|A|$ is the cardinality of $A$. 

\begin{thm}[Bourgain et.al.\cite{bourgain}]
Let $0<\eta\leq 1$ be an arbitrary number and $H$ be a multiplicative subgroup of $\mathbb{F}^{*}_{p^{n}}$ with 
$$
|H|\geq p^{n\max\{135/\eta,180000\}/\log_{2}n\log_{2}p}. 
$$
Suppose that for any proper subfield $S$ the condition $|H\cap S|\leq |H|^{1-\eta}$ holds. Then for sufficiently large $p^{n}$, there is an exponential sum estimate
$$
|\sum_{h\in H}\psi(h)|<100|H|2^{-4.5\cdot 10^{-3}(n\log_{2}p)^{0.1}},
$$
where $\psi$ is the additive character $\psi(t)=e^{2\pi i Tr(at)/p}$ for some given $a\neq 0$.
\end{thm}

From the above result we can justify that for a sufficiently large $p^{m-1}$ the subgroup $G_{p}$ is equidistributed in $\mathbb{F}_{p^{m-1}}^{*}$.

%By using the work of Arnold, Shaplinski [??, we %can show that the mapping from the $G_{p}$ to the %elements of the Galois Field %$\mathbb{F}^{*}_{p^{m-1}}$ is indeed ergodic. 

%Let $e_{p}(x)=exp(2\pi i x/p)$. 

%One can establish a correspondence 
%$$
%f_{k}^{(m)}(x)\longleftrightarrow v_{k}
%$$
%and consider the estimate on the exponential sum 
%$$
%\max_{\gamma \in %\mathbb{F}^{\times}_{p^{s}}}\left|\sum^{M}_{k=1}e_%{p}\left(Tr(\gamma %v_{k})\right)\right|<O(p^{(s-1)/2}\log p).
%$$

\section{Computing $f_{k}^{(m)}(x)$ via Linear Algebra}\label{sec_matrix}

In the previous sections we have seen that the degenerate Bernoulli polynomial $f_{k}^{(m)}(x)=(-1)^{k}\tilde{\beta}_{k}(m)/k!$ is computed by means of taking the remainder of the power of the polynomial $((1/m)x\Psi'_{m}(m))$ with $\Psi_{m}(x)$. It is well known that one can formulate these operations in terms of matrix representation.  

In order to obtain a matrix realization of the process involved in computing $f_{k}^{(m)}(x)$ notice that the following recurrence relation holds
$$
f_{k+1}^{(m)}(x)=(\frac{1}{m}x\Psi'_{m}(x)f_{k}^{(m)}(x)) \textnormal{ rem }\Psi_{m}(x)
$$
and $f^{(m)}_{0}(x)=1$. For convenience we will ignore the $(1/m)$ factor in the expression.

By the linearity of the remainder operator, it suffices to consider the action of multiplying by $x\Psi_{m}'(x)$ with $x^{j}$ for $j=0,\cdots,(m-2)$. We have 
$$
x^{j} \mapsto \left(x^{j+1}\Psi_{m}'(x)\right) \textnormal{ rem } \Psi_{m}(x).
$$
One can easily obtain a matrix representation of this mapping taking the set $\{1,x,\cdots,x^{m-2}\}$ as the basis.

\begin{exmp}
For $m=4$ we will have a $3\times 3$ matrix 
$$
\begin{pmatrix}
-3 & 1 & 1 \\ -2 & -2 & 2 \\ -1 & -1 & -1
\end{pmatrix}.
$$
\end{exmp}

Fortunately, we can provide a more elegant formulation through Circulant matrices - thanks to the substitution rule given in Lemma \ref{subs_rule_1}. The substitution rule can be expressed in two stages by deferring the substitution. The substitution is realized as a multiplication of matrices from each stage.

\begin{rul}[Deferred Substitution Rule]\label{deferred_subs_rule_1}
While applying the Substitution Rule \ref{subs_rule_1} first substitute $x^{j}$ with $x^{j\%m}$ for all $j$ to produce polynomials of degree $m-1$. Later substitute $x^{m-1}$ factors with $-\sum^{m-2}_{j=0}x^{j}$.
\end{rul}

We first consider the product 
$$
g(x)=x\Psi'_{m}(x)f_{k}^{(m)}(x) = (x+2x^{2}+\cdots+(m-1)x^{m-1})(\alpha_{0}+\cdots+\alpha_{m-1}x^{m-1}).
$$
On application of the Deferred Substitution Rule \ref{deferred_subs_rule_1} we first replace $x^{j}$ with $x^{j\%m}$ to obtain
\begin{eqnarray}
\nonumber g(x) &=& (0\alpha_{0}+(m-1)\alpha_{1}+\cdots+\alpha_{m-1})\\
\nonumber & & +(\alpha_{0}+0\alpha_{1}+(m-1)\alpha_{2}+(m-2)\alpha_{3}+\cdots+2\alpha_{m-2})x \\
\nonumber & & +(2\alpha_{0}+1\alpha_{1}+0\alpha_{2}+(m-1)\alpha_{3}+\cdots+3\alpha_{m-2})x^{2} \\
\nonumber & & +\cdots+\\
\nonumber & & +\cdots+((m-1)\alpha_{0}+\cdots+0\alpha_{m-1})x^{m-1}
\end{eqnarray}
The above product can be expressed in terms of a Circulant matrix as follows 
$$
C_{m} =\left(\begin{array}{cccccc}
0 & m-1 & m-2&\cdots&2&1\\
1 & 0 & m-1 &\cdots&3 & 2\\
2 & 1 &0 &\cdots &4& 3\\
\vdots&\vdots & \vdots&\ddots&\vdots& \vdots\\
m-2 & m-3 & m-4 &\cdots&0 & m-1\\
m-1 & m-2 & m-3 &\cdots &1& 0\\
\end{array}\right)_{m\times m}
$$
This is a right circulant matrix, in which each row is obtained from the preceding row by a cyclic shift of the entries to the right. 

Now, we apply the second substitution $x^{m-1}$ with $-\sum^{m-2}_{j=0}x^{j}$ and retaining $x^{j}$ fixed for $j\neq m-1$ to obtain the matrix 
$$
R_{m}=\begin{pmatrix}
1 & 0 & 0 & \cdots & 0 & -1 \\
0 & 1 & 0 & \cdots & 0 & -1 \\
\vdots & \cdots & \ddots & \cdots & 0 & -1 \\
0 & 0 & 0 & \cdots & 1 & -1 \\
0 & 0 & 0 & \cdots & 0 & 0
\end{pmatrix}_{m\times m}.
$$

\begin{defn}[Degenerate Bernoulli Matrix] 
Given two integers $m>0$ and $k\geq 0$, we define the degenerate Bernoulli matrix as the $$\mathcal{B}_{m}=-mR_{m}C_{m}.$$
\end{defn}

Note the $m$ factor which we skipped at the beginning for simplifying the discussion. The degenerate Bernoulli number is given by 
$$
\tilde{\beta}_{k}(m)=\frac{1}{k!}\mathcal{B}_{m}^{k}e_{1}, 
$$
where $e_{1}=(1,0,\cdots,0)^{T}$.

%\begin{lem}
%Let $A$ and $B$ be conformable partitioned matrices given as $$
%A =\left[
%\begin{array}{c|c}
%A_{11} & A_{12} \\ \hline
%A_{21} & A_{22}
%\end{array}\right],\ 
%B =\left[
%\begin{array}{c|c}
%B_{11} & B_{12} \\ \hline
%B_{21} & B_{22}
%\end{array}\right]
%$$
%Then 
%$$ 
%AB = \left[
%\begin{array}{c|c}
%A_{11}B_{11}+A_{12}B_{21} & A_{11}B_{12}+A_{12}B_{22} \\ \hline
%A_{21}B_{11}+A_{22}B_{21} & A_{21}B_{12}+A_{22}B_{22} \\
%\end{array}\right]
%$$
%\end{lem}

One can further simplify the multiplications by the following lemma. 
\begin{lem} \label{rcr}
$R_{m}C_{m}R_{m} = R_{m}C_{m}$.
\end{lem}

\begin{proof}

We write the matrices $R_{m}$ and $C_{m}$ in block matrix form 
$$
R_{m}=\left(
\begin{array}{c|c}
I_{(m-1)\times (m-1)} & -\mathbbmtt{1}_{(m-1)\times 1} \\ \hline
O_{1\times (m-1)} & 0
\end{array}\right)
$$
and 
$$C_{m}=\left(
\begin{array}{c|c}
A& X \\ \hline
Y & 0
\end{array}\right)$$
where $\mathbbmtt{1}_{(m-1)\times 1} $ is a column matrix with all entries being 1, $Y = \begin{pmatrix}
m-1&m-2&\ldots&1
\end{pmatrix},$ $X =(1,2,\cdots,m-1)^{T}, $ and
$$ A = \begin{pmatrix}
0 & m-1 & m-2&\cdots&2\\
1 & 0 & m-1 &\cdots&3\\
2 & 1 &0 &\cdots &4\\
\vdots&\vdots & \vdots&\ddots&\vdots\\
m-2 & m-3 & m-4 &\cdots&0 \\ 
\end{pmatrix},
$$ 
By a direct computation we have 
$$R_{m}C_{m} = \left(
\begin{array}{c|c}
I & -\mathbbmtt{1} \\ \hline
O & 0
\end{array}\right)
\left(
\begin{array}{c|c}
A& X \\ \hline
Y & 0
\end{array}\right)
= \left(
\begin{array}{c|c}
A-\mathbbmtt{1}Y& X \\ \hline
O & 0
\end{array}\right)$$
and 
$$R_{m}C_{m}R_{m} = \left(
\begin{array}{c|c}
A-\mathbbmtt{1}Y& X \\ \hline
O & 0
\end{array}\right)
\left(
\begin{array}{c|c}
I & -\mathbbmtt{1} \\ \hline
O & 0
\end{array}\right)
=\left(
\begin{array}{c|c}
A-\mathbbmtt{1}Y& -A\mathbbmtt{1}+\mathbbmtt{1}Y\mathbbmtt{1} \\ \hline
O & 0
\end{array}\right).
$$
It is easy to see that $$\mathbbmtt{1}Y\mathbbmtt{1}=\frac{m(m-1)}{2}\mathbbmtt{1} \textnormal{ and } A\mathbbmtt{1}=\frac{m(m-1)}{2}\mathbbmtt{1}-X.$$
Thus
$-A\mathbbmtt{1}+\mathbbmtt{1}Y\mathbbmtt{1} = X$ and hence
$R_{m}C_{m}R_{m}=R_{m}C_{m}.$
\end{proof}

It is easy to see that $R_{m}^{2}=R_{m}$. Therefore, by the above lemma we can easily deduce $(R_{m}C_{m}R_{m})^{k}=R_{m}C_{m}^{k}$. Indeed, inductively arguing on $k$, we have by Lemma \ref{rcr}
\begin{eqnarray}
\nonumber (R_{m}C_{m})^{k+1}&=&(R_{m}C_{m})(R_{m}C_{m})^k = (R_{m}C_{m})R_{m}C_{m}^{k}\\
\nonumber &=& (R_{m}C_{m}R_{m})C^{k+1}_{m}=R_{m}C^{k+1}_{m}.
\end{eqnarray}

By the above discussion we can see the following result. 
\begin{thm}
The $k^{th}$ degenerate Bernoulli number is given by the $t^{k}$ term in $R_{m}e^{-tmC_{m}}$, where we use the notation of exponentiation of a matrix $e^{tB}=\sum_{t=0}^{\infty}\frac{t^{k}}{k!}B^{k}$.
\end{thm}

A special feature of the circulant matrix is that they are diagonalizable  with the discrete fourier transform $F$ \cite{davis}. We can write 
$$
C_{m}=F^{-1}D_{m}F
$$
so that $R_{m}C^{k}_{m}=R_{m}F^{-1}D^{k}_{m}F$. 
The exponential form is given by $R_{m}F^{-1}_{m}e^{-tmD_{m}}F$. 
From the diagonalization of the circulant matrices we can easily see that the eigenvalues $\lambda_{j}$ of $C_{m}$ are given by 
$$
\lambda_{j}=\sum_{i=0}^{m-2} i\omega_{m}^{ij},
$$
where $\omega_{m}\neq 1$ and $\omega_{m}^{m}=1$. 

The matrix representation for the negative index is simpler as the recursive formula is simpler 
$$
f_{-k-1}^{(m)}(x)=(1-x)f_{-k}^{(m)}(x) \textnormal{ rem } \Psi_{m}(x).
$$
The representation is again a right circulant matrix given by $R_{m}\tilde{C}_{m}$, where  
$$
\tilde{C}_{m}=\begin{pmatrix}
1 & -1 & 0 & \cdots & 0 & 0 \\
0 & 1 & -1 & \cdots & 0 & 0 \\
\vdots &  & \ddots & \cdots &  & \vdots \\
0 & 0 & 0 & \cdots & 1 & -1 \\
-1 & 0 & 0 & \cdots & 0 & 1 \\
\end{pmatrix}_{m\times m},
$$
whose eigenvalues are 
$$
\tilde{\lambda}_{j}=\omega_{m}^{j^{2}}-\omega_{m}^{j(j+1)},
$$
where $\omega_{m}\neq 1$ satisfying $\omega_{m}^{m}=1$.

\section{Palindrome and Anti-Palindrome Property}

In this section, we show that the degenerate Bernoulli polynomial $\tilde{\beta}_{k}(m,x)$ has palindromic properties if $m$ is odd. A consequence of this is that the degenerate Bernoulli numbers vanish for certain $k$ depending on $m$.

\begin{defn}[Palindromic Polynomial \cite{robins,Zeilberger}]
Let $p(x)=\alpha_{r}x^{r}+\cdots+\alpha_{s}x^{s}$ be a non-zero polynomial, with real coefficents and $\alpha_{r},\alpha_{s}\neq 0$. The darga of $p$ is $\textnormal{dg}(p)=r+s$. We say that a polynomial $p$ of darga $n$ is palindromic if $x^{n}p(1/x)=p(x)$ i.e., $\alpha_{j}=\alpha_{n-j}$ for all $j$. We say that the polynomial $p$ is anti-palindromic if $x^{n}p(-1/x)=p(x)$ i.e., $\alpha_{j}=-\alpha_{n-j}$ for all $j$. Collection of palindromic and anti-palindromic polynomials of darga $n$ are denoted by $\mathscr{P}^{(n)}$ and $\mathscr{A}^{(n)}$ respectively.
\end{defn}

Note that, if the zero polynomial is included, we could consider both $\mathscr{P}^{(n)}$ and $\mathscr{A}^{(n)}$ as vector spaces over $\mathbb{R}$. So, for convenience we assume that $0$ polynomial is of any darga. It is easy to see that the set   
$$
\{x^{j}+x^{n-j}: j=0,\cdots,n\}
$$
is a basis set for $\mathscr{P}^{(n)}$ and 
$$
\{x^{j}-x^{n-j}: j=0,\cdots,n\}
$$
for $\mathscr{A}^{(n)}$. We prove the following result. 
 
\begin{thm}\label{thm_palin} 
Let $m$ be an odd integer and $d$ an integer. Then the following relation holds
$$
 \tilde{\beta}_{k}(md+2,x)\in 
 \begin{cases} 
 \mathscr{P}^{(m-2)} & \textnormal{ for } $d$ \textnormal{ even }\\
 \mathscr{A}^{(m-2)} & \textnormal{ for } $d$ \textnormal{ odd }.
 \end{cases}
$$
Thus, $\tilde{\beta}_{k}(md+2)=0$ for $d$ odd. 
\end{thm}
The proof of this result relies on a multiplication modulo $\Psi_{m}(x)$ operator $M_{m}$ on polynomials of degree $(m-2)$ by 
$$
M_{m}(p(x))= \left((1-x)^{m}p(x)\right) \textnormal{ rem }\Psi_{m}(x).
$$

\begin{lem} \label{lem_palin}
For $m$ an odd integer, the linear map
$$
M_{m}:\mathscr{P}^{(m-2)}\rightarrow \mathscr{A}^{(m-2)}
$$
is an isomorphism.  

\end{lem}
\begin{proof}
It is easy to show that $M_{m}$ is a linear operator modulo $\Psi_{m}(x)$. Thus it suffices to prove that on a basis element $x^{j}+x^{m-2-j}$, for $0\leq j \leq m-2$, the mapping $M_{m}$ yields an anti-palindromic polynomial of darga $(m-2)$.

By the substitution rule (\ref{subs_rule_1}) and a rearrangement of terms the following equations hold modulo $\Psi_{m}(x)$:
\begin{eqnarray}
\nonumber x^{j}(1-x)^{m}&=& \sum^{m-j-2}_{\alpha=0}(-1)^{\alpha}\begin{pmatrix}m\\ \alpha\end{pmatrix}x^{j+\alpha}\\
\nonumber & &+(-1)^{m-j-1}\left(\begin{pmatrix}m\\ m-j-1\end{pmatrix}(-1-x-\cdots-x^{m-2})\right)\\
\nonumber & &+ \sum_{\beta=0}^{j+1} (-1)^{m-j+\beta}\begin{pmatrix}m\\ m-j+\beta\end{pmatrix}x^{\beta}
\end{eqnarray}

and 

\begin{eqnarray}
\nonumber x^{m-2-j}(1-x)^{m}&=& \sum^{j}_{\beta=0}(-1)^{j-\beta}\begin{pmatrix}m\\ j-\beta\end{pmatrix}x^{m-2-\beta}\\
\nonumber & &+(-1)^{j+1}\left(\begin{pmatrix}m\\ m-j-1\end{pmatrix}(-1-x-\cdots-x^{m-2})\right)\\
\nonumber & &+ \sum_{\alpha=0}^{m-j-2} (-1)^{m-\alpha}\begin{pmatrix}m\\ m-\alpha\end{pmatrix}x^{m-j-\alpha-2}
\end{eqnarray}
It is given that $m$ is odd so we have $(-1)^{k}=-(-1)^{m-k}$. By the well-known combinatorial identity  
$$
\begin{pmatrix}
m\\k
\end{pmatrix}=\begin{pmatrix}
m\\m-k
\end{pmatrix}
$$
we have $x^{j}(1-x)^{m}\pm x^{m-j-2}(1-x)^{m}$ modulo $\Psi_{m}(x)$ equal to 
\begin{eqnarray}
\nonumber & &\sum_{\alpha=0}^{m-j-2}(-1)^{\alpha}\begin{pmatrix}
m\\ \alpha 
\end{pmatrix}\left(x^{j+\alpha}\mp x^{m-j-s-2}\right)+\sum_{\beta=0}^{j-1}(-1)^{j-\beta} \begin{pmatrix}
m\\ j-\beta
\end{pmatrix}(x^{m-2-\beta}\mp x^{\beta})\\
\nonumber & & +2\delta(-1)^{j}\begin{pmatrix}
m\\m-j-1
\end{pmatrix}(1+x+\cdots+x^{m-2}),
\end{eqnarray}
where $\delta=0$ if $+$ is chosen; otherwise $\delta=1$. Thus palindromic is transformed to anti-palindromic and vice-versa.

The map $M_{m}$ is one-one. To see this suppose $(1-x)^{m}p(x)\textnormal{ rem }\Psi_{m}(x)=0$ for some $p(x)\in\mathscr{P}^{(m-2)}$. Then, we can write $(1-x)^{m}p(x)=\Psi_{m}(x)q(x)$ for some polynomial $q(x)$. Since $\Psi_{m}(1)\neq 0$, this means all the roots of $\Psi_{m}(x)$ have to be shared by a polynomial of degree atmost $(m-2)$ degree. Therefore, this is possible only if $p(x)=0$.

\end{proof}

\begin{proof}[Proof of Theorem \ref{thm_palin}]
It is enough if the result is proved for $\mod p$ case for an arbitrary prime $p$, of course, relatively coprime to $m$. Further, by the periodicity of $f_{k}^{(m)}(x)$ we have to just show for the case multiplications by $(1-x)^{m}$ modulo $\Psi_{m}(x)$. By Lemma \ref{lem_palin}, 
all that remains to show that is $f_{2}^{(m)}(x)\in \mathscr{P}^{(m-2)}$ i.e., 
$
(x\Psi'_{m}(x))^{2} \textnormal{ rem } \Psi_{m}(x).
$
By the matrix representation... it is sufficient to look at the $j^{th}$ row and $(m-2-j)^{th}$ row for $j=0,\cdots,(m-2)$. The $j^{th}$ row in the matrix is given by 
$$
(j,\cdots,1,0,m-1,\cdots,(2j+2),(2j+1),\cdots,(j-1))
$$
and the $(m-2-j)^{th}$ row is given by 
$$
((m-2j),\cdots,(m-2j-2),(m-2j-3),\cdots,0,(m-1),\cdots,(m-3-j)).
$$
The difference of both the above rows give us 
$$
(\underbrace{a,\cdots,a}_{(j+1) \textnormal{ terms}},\underbrace{a-m,\cdots,a-m}_{(m-2j-2) \textnormal{ terms }},\underbrace{a,\cdots,a}_{(j+1) \textnormal{ terms }})
$$
where $a=(m-2j-2)$.
Taking dot product of the above vector with the vector $(0,1,2,\cdots,m-1)$ we get the expression
$$
\frac{(m-2j-2)m(m-1)}{2}-m\left(\frac{(m-2-j)(m-1-j)}{2}-\frac{j(j+1)}{2}\right)=0
$$
Hence the $j^{th}$ term is equal to $(m-2-j)^{th}$ term for $j=0,\cdots,(m-2)$ in the polynomial which means it is a palindrome of darga $(m-2)$. 
\end{proof}

Thus we have given an algebraic proof for the following result. Analytic proof was given by Young \cite{young}.

\begin{cor}
$$
\tilde{\beta}_{md+2}(m)=0 \textnormal{ for } m,d \textnormal{ odd. }
$$
\end{cor}

\section{Sums of Products}

Inspired by the formula (\ref{W1}) we look for a sum of products formula for the degenerate Bernoulli numbers. For any real number $r,$ the falling factorial defined as 
$(r)_k=r(r-1)\cdots(r-k+1)$ and  generalized falling factorial defined as $(r|\lambda)_k=r(r-\lambda)\cdots(r-(k-1)\lambda),$ for $k\in\mathbb{N}.$
For a positive integer $a,$  we denote $\sigma_k(\lambda,a)=\displaystyle\sum_{i=0}^a(i|\lambda)_k.$ We note that $$\lambda^k\sigma_k(1/\lambda,a)=\displaystyle\sum_{i=0}^a(\lambda i|1)_k=\sum_{i=0}^a(\lambda i)_k.$$ Also note that $(-r)_k=(-1)^k\langle r\rangle_k,$ where $\langle r\rangle_k$ is a rising factorial and is defines as $r(r+1)\cdots(r+k-1).$  Further, we denote that $$\widetilde{\sigma}_k(\lambda,a) =  \lambda^k\sigma_k(1/\lambda,a).$$ \\

The following remark is useful.
\begin{rema}
\begin{enumerate}
    \item For a positive integer $a,$ the generating function expressed as below (see \cite{young}):
\begin{equation} \label{yr}
  \frac{(1+\lambda t)^{(a+1)\mu}-1}{(1+\lambda t)^{\mu}-1}=\sum_{k=0}^\infty \sigma_k(\lambda,a){t^k\over k!}  
\end{equation}

\item The following Cauchy product formulae is well-known for the exponential generating functions:
    $$
    \left(\sum_{k=0}^\infty a_k {x^k\over k!}\right)\left(\sum_{k=0}^\infty b_k {x^k\over k!}\right)=\left(\sum_{k=0}^\infty c_k {x^k\over k!}\right),
    $$
    where $c_k=\displaystyle\sum_{k_1+k_2=k}\binom{k}{k_1,k_2}a_{k_1}b_{k_2}.$ 
    \item For any non-zero rational number $r,$ the Taylor's series expansion about $x=1$ is given as 
\begin{equation}\label{sp5}
  x^r=\sum_{k=0}^\infty(-1)^k(r)_k\frac{(1-x)^k}{k!},  
\end{equation}
%where $(r)_k=r(r-1)(r-2)\cdots(r-(k-1)).$
  \end{enumerate}
\end{rema}

We have the following corollary.
\begin{cor}
For positive integers $a,\ b,\ n_1$ and $n_2,$ the following hold:

 \[\Psi_a(x^{n_1})=\displaystyle\sum_{k=0}^\infty(-1)^k\widetilde{\sigma}_k(n,a-1){(1-x)^k\over k!}\]
and
\[x^{-bn_1}{\Psi_{a}(x^{n_2})} =\sum_{k=0}^\infty \binom{k}{k_1, k_2}(-1)^{k_2}\langle bn_1\rangle_{k_1}\widetilde{\sigma}_{k_2}(n_1,a-1){(1-x)^k\over k!} \]

\end{cor}
\begin{proof}
First equation is immediate if we write $(1+\lambda t) = x$ and $\lambda = {1\over n}$ in the equation (\ref{yr}). 

By using Taylor series expansion (\ref{sp5}), we have
\begin{align*}
  x^{-bn_1}&=\sum_{k=0}^\infty(-1)^k(-bn_1)_k\frac{(1-x)^k}{k!}\\
  &=\sum_{k=0}^\infty\langle bn_1\rangle_k \frac{(1-x)^k}{k!}
\end{align*}
Then it is easy to see that the second equation holds using Cauchy product for exponential generating functions for $x^{-bn_1}$ and $\Psi_a(x^{n_2}).$ 
\end{proof}

Now, we have sums of product formulae
\begin{thm}\label{sp2}
Let $m\ge2$ and $n\ge3$ be two integers that are co-prime. Assume that $a,b,\alpha$ and $\beta$ are smallest possible positive integers such that $an-bm=1$ and $\alpha m-\beta n=1.$ For $k\in\mathbb{N},$ 
if $a>\beta$ then we have
\begin{multline*}
  \sum_{k_1+k_2=k} {\widetilde{\beta}_{k_1}(m)\over k_1!}{\widetilde{\beta}_{k_2}(n)\over k_2!} = n\sum_{k_1+k_2=k}\sum_{s_1+s_2=k_1} \frac{\left( -bm\right)_{s_1}}{s_1!}\frac{\widetilde{\sigma}_{s_2}(n,a-1)}{s_2!}\frac{\widetilde{\beta}_{k_2}(m)}{k_2!}\\
  -m\sum_{k_1+k_2=k}\sum_{s_1+s_2=k_1}\frac{\left( -bm\right)_{s_1}}{s_1!}\frac{\widetilde{\sigma}_{s_2}(m,b-1)}{s_2!}\frac{\widetilde{\beta}_{k_2}(n)}{k_2!}
\end{multline*}
else if $a<\beta$
\begin{multline*}
  \sum_{k_1+k_2=k} {\widetilde{\beta}_{k_1}(m)\over k_1!}{\widetilde{\beta}_{k_2}(n)\over k_2!} =m\sum_{k_1+k_2=k}\sum_{s_1+s_2=k_1}\frac{\left( -\beta n\right)_{s_1}}{s_1!}\frac{\widetilde{\sigma}_{s_2}(m,\alpha-1)}{s_2!}\frac{\widetilde{\beta}_{k_2}(n)}{k_2!}\\ - n\sum_{k_1+k_2=k}\sum_{s_1+s_2=k_1} \frac{\left(- \beta n\right)_{s_1}}{s_1!}\frac{\widetilde{\sigma}_{s_2}(n,\beta-1)}{s_2!}\frac{\widetilde{\beta}_{k_2}(m)}{k_2!}
\end{multline*}

\end{thm}
\begin{proof}
We prove the case $an-bm=1$ with $a>\beta.$ Now consider the generating function
\begin{align*}
\frac{x-1}{\Psi_m(x)\Psi_n(x)}&=\frac{x^{an-bm}-1}{\Psi_m(x)\Psi_n(x)}\\
&=\frac{x^{-bm}(x^{an}-x^{bm})}{\Psi_m(x)\Psi_n(x)}\\
&=\frac{x^{-bm}(x^{an}-1)}{\Psi_m(x)\Psi_n(x)}-\frac{x^{-bm}(x^{bm}-1)}{\Psi_m(x)\Psi_n(x)}\\
& = \frac{x^{-bm}(x-1)\Psi_a(x^n)}{\Psi_m(x)}-\frac{x^{-bm}(x-1)\Psi_b(x^m)}{\Psi_n(x)}\\
\end{align*}
Therefore
\begin{equation}\label{ps6}
  \frac{1}{\Psi_m(x)\Psi_n(x)}=\frac{x^{-bm}\Psi_a(x^n)}{\Psi_m(x)}-\frac{x^{-bm}\Psi_b(x^m)}{\Psi_n(x)}  
\end{equation}

We note that
\begin{align*}
\frac{1}{\Psi_m(x)\Psi_n(x)}&=\left(\sum_{k=0}^\infty(-1)^{k}{\widetilde{\beta}_{k}(m)\over m}\frac{(1-x)^k}{k!}\right)\left(\sum_{k=0}^\infty(-1)^{k}{\widetilde{\beta}_{k}(n)\over n}\frac{(1-x)^k}{k!}\right)\\
&=\sum_{k=0}^\infty(-1)^{k}\sum_{k_1+k_2=k} \binom{k}{k_1,k_2}{\widetilde{\beta}_{k_1}(m)\over m}{\widetilde{\beta}_{k_2}(n)\over n}\frac{(1-x)^k}{k!},
\end{align*}

\begin{multline*}
{x^{-bm}\Psi_a(x^n)\over \Psi_m(x)}=\left(\sum_{k=0}^\infty \sum_{k_1+k_2=k}\binom{k}{k_1, k_2}(-1)^k
\left(- bm\right)_{k_1}\widetilde{\sigma}_{k_2}(n,a-1){(1-x)^k\over k!} \right) \times\\
\left(\sum_{k=0}^\infty(-1)^{k}{\widetilde{\beta}_{k}(m)\over m}\frac{(1-x)^k}{k!}  \right) 
\end{multline*}
\[ \hspace*{2cm}= {1\over m}\sum_{k=0}^\infty\sum_{k_1+k_2=k}\sum_{s_1+s_2=k_1} k!(-1)^{k}\frac{\left(- bm\right)_{s_1}}{s_1!}\frac{\widetilde{\sigma}_{s_2}(n,a-1)}{s_2!}\frac{\widetilde{\beta}_{k_2}(m)}{k_2!}\frac{(1-x)^k}{k!}\]
Similarly, we have
\[{x^{-bm}\Psi_b(x^m)\over \Psi_n(x)}={1\over n}\sum_{k=0}^\infty \sum_{k_1+k_2=k}\sum_{s_1+s_2=k_1}k!(-1)^{k}\frac{\left(- bm\right)_{s_1}}{s_1!}\frac{\widetilde{\sigma}_{s_2}(m,b-1)}{s_2!}\frac{\widetilde{\beta}_{k_2}(n)}{k_2!}\frac{(1-x)^k}{k!}\]

Now collect the coefficients of $\displaystyle{(1-x)^k\over k!}$ in the equation (\ref{ps6}), we get the desire result.
%########################################
% We can ignore if we want
\begin{multline*}
  \sum_{k_1+k_2=k} {\widetilde{\beta}_{k_1}(m)\over k_1!}{\widetilde{\beta}_{k_2}(n)\over k_2!} = n\sum_{k_1+k_2=k}\sum_{s_1+s_2=k_1} \frac{\left(- bm\right)_{s_1}}{s_1!}\frac{\widetilde{\sigma}_{s_2}(n,a-1)}{s_2!}\frac{\widetilde{\beta}_{k_2}(m)}{k_2!}\\
  -m\sum_{k_1+k_2=k}\sum_{s_1+s_2=k_1}\frac{\left(- bm\right)_{s_1}}{s_1!}\frac{\widetilde{\sigma}_{s_2}(m,b-1)}{s_2!}\frac{\widetilde{\beta}_{k_2}(n)}{k_2!}
\end{multline*}
%#############################################
The other case $\alpha m -\beta n =1$ with $a < \beta$ can be proved similarly. 
\end{proof}

\bibliographystyle{amsplain}

\end{document}